\documentclass[leqno]{amsart}

\usepackage{latexsym,amssymb,amsthm,amsmath,amscd}

\theoremstyle{plain}  
\newtheorem{theorem}{Theorem}[section]

\numberwithin{equation}{section}

 \newtheorem*{Theorem A}{{\bf Theorem A}}
\newtheorem{corollary}{Corollary}[section]
\newtheorem{lemma}{Lemma}[section]

\numberwithin{equation}{section}

\theoremstyle{remark}
\newtheorem{remark}{Remark}[section]

 \numberwithin{equation}{section}
\def\({\left( }
\def\){\right)}
\def\e{\eqref}

\begin{document}

\title[homogeneous Monge-Amp\`ere equations] {Solutions to homogeneous Monge-Amp\`ere equations of homothetic functions and their applications to production models in economics}

\author[B.-Y. Chen]{Bang-Yen Chen}

\address{Department of Mathematics,
	Michigan State University, East Lansing, Michigan
48824--1027, USA}
\email{bychen@math.msu.edu}

\begin{abstract}   Mathematically, a homothetic function is a function of the form
$f({\bf x})=F(h(x_1,\ldots,x_n))$,
where $h$ is a homogeneous function of any degree $d\ne 0$ and $F$ is a monotonically increasing function. 
In economics homothetic functions are production functions whose marginal technical rate of substitution is homogeneous of degree zero.

In this paper we classify homothetic functions satisfying the homogeneous Monge-Amp\`ere equation. Several applications to production models in economics will also be given.

\end{abstract}



\maketitle

\vskip.2in

\section{Introduction}

The original form of the $n$-dimensional Monge-Amp\`ere equation is as follow (cf. \cite{GT,Z}):
\begin{align}\label{1} \det (f_{ij})=\eta(f,f_{i},x_{i})\;\; (1\leq i,j\leq n), \end{align}
where $x_{1},\ldots,x_{n}$ are coordinates and $f_{i}=\frac{\partial f}{\partial x_{i}},\, f_{ij}=\frac{\partial^{2}f}{\partial x_{i}\partial x_{j}}$ are partial derivatives.  Throughout this paper functions are assumed to be twice differentiable.

Monge-Amp\`ere equations arise naturally in several problems in Riemannian geometry, conformal geometry, and CR geometry. One of the simplest of these applications is to the problem of prescribed Gauss curvature.
Such differential equations were first studied by G. Monge in 1784 and later by A.-M. Amp\`ere in 1820.

The Monge-Amp\`ere equation \e{1} is called {\it homogeneous} if $\eta=0$. In such case the graph 
\begin{align}\label{2} G(f)=\big(x_{1},\ldots,x_{n},f(x_{1},\ldots,x_{n})\big)\end{align}
in a Euclidean $(n+1)$-space $\mathbb E^{n+1}$ has null Gauss-Kronecker curvature (see \cite{R}).

Let $\mathbb R$ denote the set of real numbers. We put 
\begin{align}&\notag {\mathbb R}_{+}=\{r\in {\mathbb R}:r>0\}\;\;{\rm and}\;\; {\mathbb R}^{n}_{+}=\{(x_{1},\ldots,x_{n})\in {\mathbb R}^{n}:x_{1},\ldots,x_{n}>0\}.\end{align}
In economics, a {\it production function} is a 
 function $f$ from a domain $D$  of ${\mathbb R}^{n}_{+}$ into ${\mathbb R}^+$ which has non-vanishing first derivatives. 

 Almost all economic theories presuppose a production function, either on the firm level or the aggregate level. In this sense, the production function is one of the key concepts of mainstream neoclassical theories.
By assuming that the maximum output technologically possible from a given set of inputs is achieved, economists using a production function in analysis are abstracting from the engineering and managerial problems inherently associated with a particular production process.

There are two special classes of production functions that are often analyzed in economics; namely, homogeneous and homothetic production functions (cf. \cite{Mi}). A function $ f(x_{1},\cdots,x_{n})$ is called  {\it homogeneous of degree} $d$ or $d$-{\it homogeneous} if  
  \begin{align}\label{1.1}f(tx_{1},\ldots,tx_{n}) = t^{d}f(x_{1},\ldots,x_{n}).\end{align}
 A homogeneous function of degree one is simply called  {\it linearly homogeneous}.  

A {\it homothetic  function} is a production function of the form:  
\begin{align}\label{1.2} Q({\bf x})=F(h(x_1,\ldots,x_n)),\end{align}
where $h(x_1,\ldots,x_n)$ is a homogeneous function of any given degree $d\ne 0$ and $F$ is a monotonically increasing function.

In economics, an {\it isoquant} is a contour line drawn through the set of points at which the same quantity of output is produced while changing the quantities of two or more inputs.  Isoquants are also called equal product curves.
 
 While an indifference curve mapping helps to solve the utility-maximizing problem of consumers, the isoquant mapping deals with the cost-minimization problem of producers. Isoquants are typically drawn on capital-labor graphs, showing the technological tradeoff between capital and labor in the production function, and the decreasing marginal returns of both inputs. 
Homothetic functions are exactly functions whose marginal technical rate of substitution (the slope of the isoquant) is homogeneous of degree zero. Due to this, along rays coming from the origin, the slopes of the isoquants will be the same \cite{c2012}. 

In this paper we classify homothetic functions satisfying the homogeneous Monge-Amp\`ere equation. Several applications to production models in economics will also be given in this paper.

\section{Homogeneous Monge-Amp\`ere equation of homothetic functions}

 If $h(x_1,\ldots,x_n)$ is a $d$-homogeneous function with $d\ne 0$,  it follows from
the Euler Homogeneous Function Theorem  that the homogeneous function $h$ satisfies
      \begin{align}\label{2.1} x_{1}h_{1}+x_{2}h_{2}+\cdots+x_{n} h_{n}=d h.\end{align}  
      After taking the partial derivatives of \e{2.1} with respect to $x_{1},\ldots,x_{n}$, respectively, we obtain
  \begin{equation} \begin{aligned}\label{2.2} &x_{1}h_{11}+x_{2}h_{12}+\cdots+x_{n} h_{1n}=(d-1) h_{1},
  \\& x_{1}h_{12}+x_{2}h_{22}+\cdots+x_{n} h_{2n}=(d-1) h_{2},
  \\& \hskip1in \vdots
  \\ &x_{1}h_{1n}+x_{2}h_{2n}+\cdots+x_{n} h_{nn}=(d-1) h_{n}.\end{aligned}  \end{equation}

If $d=1$, it follows from \e{2.2} and Cramer's rule that 
\begin{align}\label{2.3} x_{1}\det(h_{ij})=\cdots=x_{n}\det(h_{ij})=0.\end{align}

Since \e{2.3} holds for any $x_{1},\ldots,x_{n}$ whenever $h$ is defined, $h$ must satisfies  $\det(h_{ij})=0$. Therefore 
we have the following well-known lemma.
 
 \begin{lemma}\label{L} Every linearly homogeneous function $h$ satisfies the homogeneous Monge-Amp\`ere equation $\det(h_{ij})=0$.
\end{lemma}

Now we  give the following.

\begin{theorem}\label{T:2.1} Let  $h({\bf x})$ be a homogeneous function with degree $d\ne 1$. If $h$ satisfies the homogeneous Monge-Amp\`ere equation $\det(h_{ij})=0$, then for every function $F$ with $F'\ne 0$ the homothetic function $f({\bf x})=F(h({\bf x}))$ satisfies 
the homogeneous Monge-Amp\`ere equation: $\det(f_{ij})=0$.
\end{theorem}
\begin{proof} Let  $h$ be a homogeneous function with degree $d\ne 1$. Since $f({\bf x})=F(h({\bf x}))$, we have
\begin{align}\label{2.4} f_{i}=F'(u)h_{i},\;\;   f_{ij}= F'h_{ij}+F'' h_{i}h_{j}, \;\; 1\leq i,j\leq n.\end{align} 
It follows from \e{2.4} that
\begin{align} &\label{2.5}  \det(f_{ij})=(F'(u))^{n}\Bigg\{\det(h_{ij})F'(u)+ F''(u)\sum_{i,j=1}^{n} h_{i} h_{j}H_{ij}\Bigg\},\end{align}
where  $H_{ij}=(-1)^{i+j} M_{ij}$ is the cofactor of $h_{ij}$ and $M_{ij}$ is the  minor of $h_{ij}$ for the Hessian matrix $(h_{ij})$. 

Since $h$ is a homogeneous function with degree $d\ne 1$, \e{2.2} gives
 \begin{equation} \begin{aligned}\label{2.6} &h_{1}=\frac{x_{1}h_{11}+x_{2}h_{12}+\cdots+x_{n} h_{1n}}{d-1},
   \\& \hskip.5in \cdots
  \\ &h_{n}=\frac{x_{1}h_{1n}+x_{2}h_{2n}+\cdots+x_{n} h_{nn}}{d-1}.\end{aligned}  \end{equation}
 After substituting \e{2.6} into \e{2.5} we find
 \begin{equation}\begin{aligned} &\label{2.7}  \det(f_{ij})=(F'(u))^{n}\Bigg\{ \det(h_{ij}) F'(u) +\frac{F''(u)}{(d-1)^{2}}\sum_{i,j,k,\ell=1}^{n} x_{k}x_{\ell} h_{ik}h_{j\ell}H_{ij}\Bigg\}.
 \end{aligned}\end{equation}
Now, by applying the Cofactor Expansion Formula for determinants to \e{2.7} and then using \e{2.2} and \e{2.1}, we derive  that
\begin{equation}\begin{aligned} \label{2.8}  \det(f_{ij})&= \det(h_{ij})(F'(u))^{n}\left\{  F'(u) +\frac{F''(u)}{d-1}\sum_{i=1}^{n} x_{i}h_{i}\right\}
\\&= \det(h_{ij})\frac{(F'(u))^{n}}{d-1}\left\{ (d-1) F'(u) +d h F''(u)\right\}. \end{aligned}\end{equation}
Consequently, if $h$ satisfies the homogeneous Monge-Amp\`ere equation $\det(h_{ij})=0$, then  $f=F\circ h$ satisfies  the homogeneous Monge-Amp\`ere equation $\det(f_{ij})=0$.
\end{proof}

Conversely, in views of Lemma \ref{L} and Theorem \ref{T:2.1} we give the following.

\begin{theorem}\label{T:2.2} Let $F(u)$ be a function with $F'\ne 0$ and $u=h({\bf x})$ be a homogeneous function with degree $d\ne 1$. If  $f=F\circ h$ satisfies 
the homogeneous Monge-Amp\`ere equation $\det(f_{ij})=0$, then either 
\vskip.05in

 {\rm (i)} the inner function $h$ satisfies $\det(h_{ij})=0$ or 

\vskip.05in

{\rm (ii)} up to constants, $f=F\circ h$ is a linearly homogeneous function.
\end{theorem}
\begin{proof} Let $F(u)$ be a twice differentiable function with $F'\ne 0$ and $h$ is a homogeneous function with degree $d\ne 1$. Then we have \e{2.8}. Suppose that $f=F\circ h$ satisfies 
the homogeneous Monge-Amp\`ere equation $\det(h_{ij})=0$. Then it follows from \e{2.8} that either $ \det(h_{ij})=0$ or $F$ satisfies
\begin{align} \label{2.9} d u F''(u)+(d-1) F'(u) =0.\end{align}

If $\det(h_{ij})=0$ holds, we obtain case (i). Otherwise, after solving \e{2.9} we obtain $F(u)=\alpha u^{\frac{1}{d}}+\beta$ for some real numbers $\alpha,\beta$ with $\alpha\ne 0$. Since $h$ is a homogeneous function of degree $d$, so up to constants $f=F\circ h$ is a linearly homogeneous function.
\end{proof}

\section{Two-input homothetic functions satisfying Monge-Amp\`ere equation}

The next theorem completely classifies two-input homothetic functions satisfying the homogeneous Monge-Amp\`ere equation.

\begin{theorem}\label{T:3.1}  Let $F$ be a  function with $F'\ne 0$ and  $h(x,y)$ be a homogeneous function. Then  $f=F(h(x,y))$ satisfies  the homogeneous Monge-Amp\`ere equation $\det(f_{ij})=0$ if and only if either
 \begin{enumerate}

\item the inner function $h$  is of the form $(ax+by)^{d}$ for some constants $a,b$, or
 
\item  up to constants $f(x,y)$ is a linearly homogeneous function.
\end{enumerate} 
\end{theorem}
\begin{proof}  Let $F$ be a function with $F'(u)\ne 0$ and let $u=h(x,y)$ be a homogeneous function. Assume that $f=F\circ h$ satisfies  the homogeneous Monge-Amp\`ere equation $\det(f_{ij})=0$. Then  it follows from  \e{2.5} that $h$ and $F$ satisfy
\begin{align} &\label{3.1}  0=\det(h_{ij})F'(u)+ F''(u)(h_{1}^{2} h_{22}+h_{2}^{2}h_{11}-2 h_{1}h_{2}h_{12}),\end{align}
where $h_{1}=h_{x},h_{2}=h_{y},h_{11}=h_{xx}$ etc.

  \vskip.05in
 {\it Case} (a): $d=1$.  Lemma \ref{L} implies that $\det(h_{ij})=0$.
 Hence  equation \e{3.1} reduces to  \begin{align} \label{3.2} 0=(h_{1}^{2}h_{22}+h_{2}^{2} h_{11}-2 h_{1} h_{2}h_{12})F''(u).\end{align}

If $F''=0$, then $F(u)=\alpha u+\beta$ for some real numbers $\alpha,\beta$ with $\alpha\ne 0$. Thus, up to the  constant $\beta$, $f(x,y)=\alpha h(x,y)$ which is a linearly homogeneous function. Thus we obtain case (2) of the theorem.

Next, let us assume that $F''\ne 0$. Then  \e{3.1} yields
 \begin{align} \label{3.3}h_{1}^{2}h_{22}+h_{2}^{2} h_{11}-2 h_{1} h_{2}h_{12}=0.\end{align}

From \e{2.2} we get $$h_{11}=-\(\frac{y}{x}\)h_{12},\;\;  h_{22}=-\(\frac{x}{y}\)h_{12}.$$ By substituting these into \e{3.3} we find
 \begin{align} \label{3.4}0=(xh_{1}^{2}+{y}h_{2})^{2}h_{12}=d^{2} h^{2} h_{12}.\end{align}
 Therefore  $h_{12}=0$, which implies that $h(x,y)=p(x)+q(y)$ for some functions $p(x),\,q(y)$. Since $h(x,y)$ is linearly homogeneous, we must have $h(x,y)=ax+by$ for some real numbers $a,b$.  This gives case (1) of the theorem with $d=1$,

\vskip.05in
 {\it Case} (b): $d\ne 1$. Let us consider the functions $\hat F$ and $\hat h$ given by
 \begin{align} \hat F(u)=F(u^{d}),\;\; \hat h(x,y)=(h(x,y))^{\frac{1}{d}}.\end{align}
 Then $\hat h$ is a linear homogeneous function such that $$f(x,y)=F(h(x,y))=\hat F(\hat h(x,y)).$$ Since $f=\hat F\circ \hat h$ satisfies the 
 the homogeneous Monge-Amp\`ere equation $\det(f_{ij})=0$ and $\deg \hat h=1$, we may apply the same argument given in case (a) to conclude that either, up to constants, $f$ is linearly homogeneous or $\hat h(x,y)=a x+by$ for some constants $a,b$. Therefore in the latter case we have $h(x,y) =(ax+by)^{d}$.
 
 The converse can be verify easily. 
\end{proof}

\section{$n$-input homothetic functions satisfying Monge-Amp\`ere equation}

Theorem \ref{T:3.1} is false if $n\geq 3$. 
For example, if $\psi(y,z)$ is a linearly homogeneous function and 
\begin{align}\label{4.1}h(x,y,z)=x+\psi(y,z),\end{align} then for any function $F(u)$, the composition $f=F\circ h$ satisfies the homogeneous Monge-Amp\`ere equation $\det(f_{ij})=0$.

The following theorem determines all $n$-input homothetic functions with $n\geq 3$ which satisfy the homogeneous Monge-Amp\`ere equation.

\begin{theorem}\label{T:4.1} Let $F$ be a function with $F\ne 0$ and  $h$ an $n$-input  $d$-homogeneous function with $d\ne 0$ and $n\geq 3$. Then  $f=F\circ h$ satisfies 
the homogeneous Monge-Amp\`ere equation $\det(f_{ij})=0$ if and only if either
 \begin{enumerate}

\item up to constants $f$ is a linearly homogeneous function, or

\item $f$ is of the form $F\big(x_{1}\phi\big(\frac{x_{2}}{x_{1}},\ldots,\frac{x_{n}}{x_{1}}\big)\big)$, where $\phi(u_{2},\ldots,u_{n})$ is an $(n-1)$-input function satisfying $\det(\phi_{ij})=0$.
\end{enumerate} 
\end{theorem}
\begin{proof} Let $F$ be a function with $F'\ne 0$ and $h(x,y,z)$ be a homogeneous function of degree $d\ne 0$. 

  \vskip.05in
 {\it Case} (a): $d=1$. Since $h$ is a linearly homogeneous function, we may put
\begin{align}\label{4.2} h(x_{1},\ldots,x_{n})=x_{1} \phi\(\frac{x_{2}}{x_{1}},\ldots,\frac{x_{n}}{x_{1}}\)\end{align}
for some function $\phi$. Thus we have
\begin{align}\label{4.3} f(x_{1},\ldots,x_{n})=F\(x_{1}\phi\!\(\frac{x_{2}}{x_{1}},\ldots,\frac{x_{n}}{x_{1}}\)\! \).\end{align}
It follows from \e{4.3} that
\begin{align}\label{4.4} {x_{1}^{n-1}}\det(f_{ij})={\phi^{2} (F')^{n-1}F''}\det(\phi_{ij}).\end{align}
We conclude from \e{4.4} that if $f=F\circ h$ satisfies  $\det(f_{ij})=0$, then we have either $F''=0$ or $\det(\phi_{ij})=0$.  

If $F''=0$, then up to a suitable constant $f$ is a linearly homogeneous  function. If $\det(\phi_{ij})=0$, $\phi$ satisfies  $\det(\phi_{ij})=0$.

\vskip.05in
 {\it Case} (b): $d\ne 1$. Let us consider the functions $\hat F$ and $\hat h$ defined by
 \begin{align} \hat F(u)=F(u^{d}),\;\; \hat h(x_{1},\ldots,x_{n})=(h(x_{1},\ldots,x_{n}))^{\frac{1}{d}}.\end{align}
 Then $\hat h$ is a linear homogeneous function such that $$f(x_{1},\ldots,x_{n})=F(h((x_{1},\ldots,x_{n}))=\hat F(\hat h(x_{1},\ldots,x_{n})).$$ 
 Because $f=\hat F\circ \hat h$ satisfies
  $\det(f_{ij})=0$ and $\deg \hat h=1$ holds, we may apply the same argument as case (a) to conclude that either
  
   (i) up to constants $f$ is a linearly homogeneous function or
   
    (ii) $\hat h$ takes the form:
  $$\hat h(x_{1},\ldots,x_{n})=x_{1}\varphi \!\(\frac{x_{2}}{x_{1}},\ldots,\frac{x_{n}}{x_{1}}\)\!,$$ where $\varphi$ is an $(n-1)$-input function satisfying $\det(\varphi_{ij})=0$. Consequently, we obtain case (2) of the theorem.
 
 The converse can be verify easily. 
\end{proof}

An immediate consequence of Theorem \ref{T:4.1} is the following.

\begin{corollary}\label{C:4.1}  Let $\phi(u_{2},\ldots,u_{n})$ be a function satisfying the homogeneous Monge-Amp\`ere equation. Then, for each function $F$ with $F'\ne 0$,  the homothetic function 
$$f(x_{1},\ldots,x_{n})=F\(x_{1} \phi \!\(\frac{x_{2}}{x_{1}},\ldots,\frac{x_{n}}{x_{1}}\)\!\)$$
satisfies the the homogeneous Monge-Amp\`ere equation  $\det(f_{ij})=0$.
\end{corollary}

\begin{remark} Since there are ample examples of functions $\phi(u_{2},\ldots,u_{n})$ which satisfy the homogeneous Monge-Amp\`ere equation, Corollary \ref{C:4.1} implies that there exist infinitely many  homothetic functions satisfying the corresponding homogeneous Monge-Amp\`ere equation. \end{remark}

\section{Applications to production models in economics}

In economics, goods that are completely substitutable with each other are called perfect substitutes. They may be characterized as goods having a constant marginal rate of substitution. Mathematically,
a production function is a {\it perfect substitute} if it is of the form:
\begin{align}\label{5.1} f({\bf x})= \sum_{i=1}^{n}a_{i}x_{i}\end{align}
for some nonzero constants $a_{1},\ldots,a_{n}$. 

Since every $n$-input production function $f({\bf x})$ can be identified with its graph $$G(f)=(x_{1},\ldots,x_{n},f),$$ which represents the quantity of output. Consequently,
 many important properties of production functions in economics can be interpreted in terms of the geometry of their graphs (cf. for instance, \cite{c,c1,c2,c2012,VV,V}).
 
 In this section we provide several applications of our results to some important production models.
 
 \begin{corollary} \label{C:5.1} Let $f(x,y)=F(h(x,y))$ be a homothetic production function. Then 
 the graph of  $f$ is a flat surface if and only if either
 \begin{enumerate}

\item $f(x,y)$ is  linearly homogeneous, or
 
\item $F(u)$ is a strictly increasing function and $h(x,y)$ is a perfect substitute.
\end{enumerate} 
  \end{corollary}
 \begin{proof} Follows immediately from Theorem \ref{T:3.1}.
 \end{proof}

This corollary generalizes of a result of the author and Vilcu.

\begin{corollary}\label{C:5.2}  Let $f=F(h(x_{1},\ldots,x_{n}))$ be a homothetic function such that $h$ is a $d$-homogeneous function with $d\ne 1$. Then the graph of $f$ has null Gauss-Kronecker curvature if and only if either 
\vskip.05in

 {\rm (i)}  $h$ satisfies the homogeneous Monge-Amp\`ere equation $\det(h_{ij})=0$ or 

\vskip.05in

{\rm (ii)} up to constants, $f=F\circ h$ is a linearly homogeneous function.
\end{corollary}
\begin{proof} Follows immediately from Theorem \ref{T:2.1} and Theorem \ref{T:2.2}.
\end{proof}

In 1928  Cobb and  Douglas  introduced in \cite{CD} a famous two-input production function
 \begin{align}\label{5.2} P = bL^{k}C^{1-k},\end{align}
 where $b$  represents the total factor productivity, $P$ the total production, $L$ the labor input and $C$ the capital input.
The Cobb-Douglas production function is widely used  in  economics to represent the relationship of an output to inputs.
 Later work in the 1940s prompted them to allow for the exponents on $C$ and $L$ vary, which resulting in estimates that subsequently proved to be very close to improved measure of productivity developed at that time (see \cite{D}).

 In its generalized form Cobb-Douglas'  production function may be expressed as
\begin{align}\label{5.3} P(x_{1},\ldots,x_{n}) =\gamma x_{1}^{\alpha_{1}}\cdots x_{n}^{\alpha_{n}},\end{align}
where $\gamma$ is a positive constant and $\alpha_{1},\ldots,\alpha_{n}$ are nonzero constants. 

In 1961  Arrow,  Chenery, Minhas and  Solow  introduced in  \cite{ACMS} another two-input production function given by
\begin{align}\label{5.4}Q=F\cdot (a K^{r}+(1-a) L^{r})^{\frac{1}{r}},\end{align}
where $Q$ is the output, $F$ the factor productivity, $a$ the share parameter, $K$ and $L$ the primary production factors, $r=(s-1)/s$, and $s=1/(1-r)$ is the elasticity of substitution.
The generalized form of ACMS production function is given by
\begin{align}\label{5.5} Q({\bf x})=\gamma \Big(\sum_{i=1}^{n}a_{i}^{\rho}x_{i}^{\rho}\Big)^{\! \frac{d}{\rho}},\end{align}
where $a_{i},p,\gamma, \rho$ are nonzero constants.

By applying Theorem \ref{T:2.2} we have the following two corollaries.

\begin{corollary}\label{C:5.3} Let $P(x_{1},\ldots,x_{n})$ be a Cobb-Douglas production function given by \e{5.3} and $F$ be a strictly increasing function. Then the graph of the homothetic production function $F\circ P$ has null Gauss-Kronecker curvature if and only if both $F$ and $P$ are linear.
\end{corollary}
\begin{proof} Since $P$ is given by \e{5.3}, the homothetic production function $f=F\circ P$ cannot be linearly homogeneous unless both $F$ and $P$ are linear. In this case, up to constants, $f$ is a linear Cobb-Douglas production function.

Now, assume that $f=F\circ P$ is not a linearly homogeneous function and the graph of $f$ has null Gauss-Kronecker curvature. 

 Without loss of generality, we may assume  $\deg P\ne 1$, since otherwise we may consider $\hat F(u)=F(\sqrt{u})$ and $\hat P({\bf x})=P({\bf x})^{2}$ instead. Therefore we may apply Theorem \ref{T:2.2} to conclude that the Cobb-Douglas function satisfies $\det(P_{ij})=0$, which is impossible unless $\deg P=1$. Consequently, this case is impossible.

 The converse is easy to verify.
\end{proof}

\begin{corollary} \label{C:5.4} Let $Q(x_{1},\ldots,x_{n})$ be a ACMS production function given by \e{5.5} and $F$ be a strictly increasing function. Then the graph of the homothetic production function $F\circ Q$ has null Gauss-Kronecker curvature if and only if either
\begin{enumerate}
\item $\rho=1$, or 

\item  up to constants, $f=\gamma \left(\sum_{i=1}^{n}a_{i}^{\rho}x_{i}^{\rho}\right)^{\! \frac{1}{\rho}}$ with nonzero constants $a_{i},\gamma,\rho$.
\end{enumerate}
\end{corollary}
\begin{proof} Since $Q$ is given by \e{5.5}, the homothetic production function $f=F\circ Q$ cannot be linearly homogeneous unless $F$ and $Q$ are both linear. Thus we have case (2) of the corollary.

Now,  assume that $f$ is not a linearly homogeneous function and  $F\circ Q$ has null Gauss-Kronecker curvature.
Without loss of generality,
 we may assume  $\deg Q\ne 1$, since otherwise we may consider $\hat F(u)=F(\sqrt{u})$ and $\hat Q({\bf x})=Q({\bf x})^{2}$.  Therefore, according to Theorem \ref{T:2.2}, the ACMS function must satisfies $\det(Q_{ij})=0$, which is impossible unless $\rho=1$.
 
 The converse is easy to verify.
\end{proof}

\end{document}